%% file: Trivializations-ArxivV2.tex
\documentclass{amsart}
\usepackage{graphicx}
\usepackage{amsmath,amsthm,amsfonts,amscd,amsxtra,amsopn,amssymb,setspace,verbatim,subfigure}
\usepackage{eucal}
\usepackage{verbatim} 
\usepackage{ifthen} 
\usepackage{array}
\usepackage{caption}
\usepackage{stmaryrd}
\usepackage{pdfsync}
\usepackage{sseq}
\usepackage{upgreek}  
\usepackage{mathrsfs}
\input xy
\xyoption{all}

\usepackage{color}
\usepackage[margin=1.28in]{geometry}

\usepackage{hyperref} 
\hypersetup{ 
colorlinks,%
citecolor=black,%
filecolor=black,%
linkcolor=black,%
urlcolor=black 
} 
\input elliptic-macros-thesis.tex

\numberwithin{equation}{section}

\theoremstyle{plain}
\newtheorem{thm}[equation]{Theorem}
\newtheorem{prop}[equation]{Proposition}
\newtheorem{lemma}[equation]{Lemma}
\newtheorem{cor}[equation]{Corollary}

\theoremstyle{definition}
\newtheorem{defn}[equation]{Definition}

\newtheorem{example}[equation]{Example}

\theoremstyle{remark}
\newtheorem{rem}[equation]{Remark}

\definecolor{darkgray}{gray}{0.4}

\newcommand\phalf{\frac{p_1}{2}}
\renewcommand\o{\circ}
\newcommand\st{\>\> \Big| \>\>}
\def\Image{\operatorname{Image}}

\renewcommand\la{\langle}
\renewcommand\ra{\rangle}
\newcommand\from{\leftarrow}

\newcommand\Ob{\operatorname{Ob}}

\renewcommand\L{\Lambda}
\renewcommand\cL{\Uplambda}

\renewcommand{\cT}[1]{\mathcal{T}_{\operatorname{geo}} (#1)}
\newcommand{\cTt}[2][a]{\mathcal{T}_{\operatorname{top}}(#2)}
\renewcommand{\T}[1]{T_{\operatorname{geo}}(#1)}
\newcommand{\Tt}[2][a]{T_{\operatorname{top}}({#2})}
\newcommand{\Tc}[2][a]{T_{\operatorname{curv}}(#2)}
\newcommand{\Th}[2][a]{T_{\operatorname{dR}}(#2)}

\renewcommand{\H}[4][M]{\check{H}^{#2}(#1; #4)}
\renewcommand{\cH}[4][M]{\check{\mathcal{H}}^{#2}(#1; #4)}

\newcommand{\Cbg}[5][M]{\check{C}(#2)^{#3}(#1;#5)}
\newcommand{\Hbg}[5][M]{\check{H}(#2)^{#3}(#1;#5)}
\newcommand{\Zbg}[5][M]{\check{Z}(#2)^{#3}(#1;#5)}
\newcommand{\Cbgc}[5][M]{\check{C}(#2)^{#3}_c(#1;#5)}
\newcommand{\cHc}[4][M]{\check{\mathcal{H}}_c^{#2}(#1; #4)}

\renewcommand\l{\check{\lambda}}
\newcommand\s{\check{s}}
\renewcommand\x{\check{x}}
\renewcommand\y{\check{y}}

\newcommand\curv{\operatorname{curv}}

\begin{document}
\title{Trivializations of differential cocycles}
\author{Corbett Redden}
\address{Department of Mathematics, Michigan State University, East Lansing, MI 48824}
\email{redden@math.msu.edu}
\curraddr{Max-Planck-Institut f\"ur Mathematik, Vivatsgasse 7, 53111 Bonn, Germany}
\thanks{This work was partially supported by the NSF RTG Grant DMS-0739208.}


\begin{abstract}
Associated to a differential character is an integral cohomology class, referred to as the characteristic class, and a closed differential form, referred to as the curvature.  The characteristic class and curvature are equal in de Rham cohomology, and this is encoded in a commutative square.  In the Hopkins--Singer model, where differential characters are equivalence classes of differential cocycles, there is a natural notion of trivializing a differential cocycle.  In this paper, we extend the notion of characteristic class, curvature, and de Rham class to trivializations of differential cocycles.  These structures fit into a commutative square, and this square is a torsor for the commutative square associated to characters with degree one less.  Under the correspondence between degree 2 differential cocycles and principal circle bundles with connection, we recover familiar structures associated to global sections.
\end{abstract}

\maketitle
\thispagestyle{empty}


\section{Introduction}\label{sec:Intro}

The differential cohomology groups $\H{k}{\R}{\Z}$, defined for a smooth manifold $M$, are a hybrid of ordinary integral cohomology and differential forms.  More precisely, there is a functorially defined commutative square
\begin{equation}\label{eq:fundsquare} \xymatrix{  \H{k}{\R}{\Z} \ar[d] \ar[r] & H^k(M;\Z) \ar[d]\\
\Omega^k(M)_\Z \ar[r] &H^k(M;\R)_\Z ,} \end{equation}
where $\Omega^k(M)_\Z$ are smooth forms with $\Z$-periods, and $H^k(M;\R)_\Z \subset H^k(M;\R)$ is the image of $H^k(M;\Z) \to H^k(M;\R)$.  All four homomorphisms are surjective.  To an element in $\H{k}{\R}{\Z}$, the induced elements in $H^k(M;\Z)$, $\Omega^k(M)$, and $H^k(M;\R)$ are referred to as the characteristic class, curvature, and de Rham class, respectively.  While the characteristic class and curvature do not completely determine an element in $\H{k}{\R}{\Z}$, they provide a useful way of compartmentalizing the differential cohomology groups.

Cocycle models for differential cohomology allow one to define trivializations.  This idea is not new; for example, Hopkins--Singer emphasize this point in \cite{HS05}, and trivializations play a fundamental role in Freed's paper \cite{Fre00}.  While there are multiple notions of a trivialization, we consider the definition that generalizes the notion of trivializing a $\R/\Z$-bundle with connection by an arbitrary global section.  More generally, a differential cocycle $\x$, representing $[\x] \in \H{k}{\R}{\Z}$, will admit a trivialization if and only if the characteristic class in $H^k(M;\Z)$ vanishes.  The goal of this paper is to extend the characteristic class, curvature, and de Rham class to trivializations of differential cocycles.  This allows one to compartmentalize a trivialization into simpler objects, just as one can partially describe a differential cohomology class using the commutative square \eqref{eq:fundsquare}.

More concretely, let $\x \in \cH{k}{\R}{\Z}$ be a differential cocycle admitting a trivialization.  We let $\cT{x}$  denote the category of trivializations of $\x$, and we let $\T{x}$ denote the set of isomorphism classes of trivializations.  Both of these were previously defined in \cite{HS05}, but this notation was not used.  The subscript ``$\text{geo}$" is due to the fact that $\T{\x}$ is like an affine version of differential characters.  More precisely, $\T{\x}$ is a torsor (principal homogeneous space) for $\H{k-1}{\R}{\Z}$, and elements contain local geometric information.

Because the sets $\T{\x}$ are torsors and not groups, they are intrinsically subtle.  Furthermore, their definition involves the actual cocycle $\x$ as opposed to the differential character in $\H{k}{\R}{\Z}$.  The motivation for this paper is to provide tools and language that allow one to describe classes in $\T{\x}$ using simpler structures.

We define a category $\cTt{\x}$, whose isomorphism classes are denoted $\Tt{\x}$, along with sets $\Tc{\x}$ and $\Th{\x}$.  These sets are where the characteristic class, curvature, and de Rham class of a trivialization live.  While the sets $\T{\x}$ and $\Tt{\x}$ depend on the specific cocycle $\x$, the sets $\Tc{\x}$ and $\Th{\x}$ are defined using only the equivalence class $[\x] \in \H{k}{\R}{\Z}$.  Elements of $\Tc{\x} \subset \Omega^{k-1}(M)$ are differential forms that map to the character $[\x]$ under the natural homomorphism $\Omega^{k-1}(M) \to \H{k}{\R}{\Z}$.  Elements of $\Th{\x}$ are lifts of the differential character from $\Hom(Z_{k-1}(M), \R/\Z)$ to $\Hom(Z_{k-1}(M),\R)$.

Our notation is justified by Theorem \ref{thm:torsorsquare}, which states that the sets $T_{\bullet}(\x)$ naturally fit into a commutative square that is a torsor for the fundamental square \eqref{eq:fundsquare} associated to $\H{k-1}{\R}{\Z}$.
\[ \xymatrix{ \T{\x}  \ar[d] \ar[r]& \Tt[\l]{\x} \ar[d]\\
\Tc[\l]{\x} \ar[r] & \Th[\l]{\x}
} \qquad \qquad
 \xymatrix{ \H{k-1}{V}{\Z} \ar[d]\ar[r]&H^{k-1}(M;\Z) \ar[d]\\
\Omega^{k-1}(M)_\Z \ar[r] &H^{k-1}(M;\R)_\Z }
\]
In other words, the sets $\T{\x}$, $\Tt{\x}$, $\Tc{\x}$, $\Th{\x}$ are torsors over $\H{k-1}{\R}{\Z}$, $H^{k-1}(M;\Z)$, $\Omega^{k-1}(M)_\Z$, $H^{k-1}(M;\R)_\Z$, respectively, and there are natural surjective equivariant maps making the square on the left into a commutative square.

The paper is organized as follows.  Section \ref{sec:Background} contains standard background information on Cheeger--Simons differential characters and the Hopkins--Singer model for differential ordinary cohomology.  The one part not already in the literature is the exact sequence for the change of coefficients in differential cohomology; it appears in Proposition \ref{prop:latticechange} and \eqref{eq:ses0Z}.  Section \ref{sec:TrivOrdinary} is the heart of the paper.  We define the the sets $T_{\bullet}(\x)$ and show they fit into a commutative square of torsors.  Subsection \ref{subsec:Hodge} discusses exactly how the Hodge decomposition of forms fits into differential cohomology.  On a closed manifold, a Riemannian metric induces a right inverse $\Tc{\x} \from \Th{\x}$ that is equivariant with respect to the usual right inverse $\Omega^{k-1}(M)_\Z \from H^{k-1}(M;\R)_\Z$ given by projecting onto harmonic forms.  However, we prove this cannot be refined to give a natural right inverse for the characteristic class or curvature map.  Subsections \ref{subsec:Product} and \ref{subsec:Integration} show that the sets $T_{\bullet}(\x)$ behave as expected with respect to products and integration.  In Section \ref{sec:Examples}, we consider the degree 2 case.  When a cocycle $\x \in \cH{2}{\R}{\Z}$ represents a principal $\R/\Z$-bundle with connection $(P,\Theta) \overset{\pi}\to M$, the sets $T_{\bullet}(\x)$ naturally recover familiar structures associated to global sections of $P$.  

We should note that there exist different cocycle models for various differential cohomology theories.  For example, Deligne--Beilinson cohomology \cite{Bry93} and Harvey--Lawson spark complexes \cite{HL06} give alternate constructions of the groups $\H{*}{\R}{\Z}$.  More generally, Hopkins--Singer define a differential refinement for any generalized cohomology theory \cite{HS05}.  The commutative square of torsors, given by trivializations of a differential cocycle, generalizes to other cocycle models for differential cohomology theories.  We hope to give a broad but concise account of this in the near future.  In this paper, though, we discuss only one model so that the discussion remains both explicit and simple.

There are various reasons to consider trivializations of differential cocycles.  Many geometric structures arising in theoretical physics can be described in terms of differential (generalized) cohomology or trivializations of differential cocycles.  As explained in \cite{Fre00}, trivializations in differential ordinary cohomology are a natural generalization of abelian gauge fields in the presence of a magnetic current, and they often give a mathematical description of anomaly cancellation mechanisms.  In another example, trivializations play an important role in formulating $T$-duality within the context of twisted differential $K$-theory \cite{KahleVal09}.  Also, trivializations are useful when considering lifts of structure groups.  On a $SO$-bundle with connection, lifts to a spin$^c$-bundle with connection are equivalent to trivializations of a differential refinement of the $W_3$ characteristic class.  Similarly, on a $Spin$-bundle with connection, one can formulate geometric string structures as trivializations of a cocycle for the character $\check{\phalf}$.  We discuss both these statements in \cite{Redden12a}, and this definition of geometric string structure was given in \cite{Wal09} using the language of bundle 2-gerbes.

The author wishes to thank both Michigan State University, where most of this research was originally conducted, and the Hausdorff Research Institute for Mathematics, where most of this paper was written.  Thanks also to Alex Kahle, Arturo Prat-Waldron, Konrad Waldorf, Fei Han, and Peter Teichner for helpful comments.


\section{Review of differential cohomology}\label{sec:Background}
We always work in the category of smooth manifolds; objects are smooth manifolds and morphisms are smooth maps.  We also work with the smooth singular chain and cochain complexes, $C_*(M)$ and $C^*(M;-)$, and we denote cycles and cocycles by $Z_*(M)$ and $Z^*(M;-)$.  The bracket $\la\cdot,\cdot \ra$ denotes the pairing between cochains and chains.

Let $V$ be a finite-dimensional vector space, and $\L \subset V$ a completely disconnected subgroup; i.e. the only connected components of $\L$ are points.  Define $H^k(M;V)_\L \subset H^k(M;V)$ to be the image of $H^k(M;\L) \to H^k(M;V)$, and let $\Omega^k(M;V)_\L \subset \Omega^k(M;V)$ denote the closed smooth $V$-valued $k$-forms on $M$ with periods in $\L$.  In other words, $\omega \in \Omega^k(M;V)_\L$ if and only if $\int_z \omega \in \L$ for all smooth cycles $z\in Z_k(M)$.  Hence, $\Omega^k(M;V)_\L$ maps onto $H^k(M;V)_\L$ when quotienting by $d \Omega^{k-1}(M;V)$.

We now introduce the differential (ordinary) cohomology groups $\H{k}{V}{\L}$.  While there exist many different explicit models, they are all isomorphic, a fact proven elegantly in \cite{SS08a}.  We use only the models from Cheeger--Simons \cite{CS85} and Hopkins--Singer \cite{HS05}.  The first, based on differential characters, is the most geometric and emphasizes the idea of holonomy.  The second model, which is given as the cohomology of a cochain complex, allows us to define trivializations.  

\begin{defn}[Cheeger--Simons \cite{CS85}]\label{defn:diffchars}The group of {\it differential characters} of degree $k$,\footnote{This grading is different than the one originally used by Cheeger--Simons.}  with values in $V/\L$, is
\begin{align*} \H{k}{V}{\L} \= \left\{ \chi \in \Hom(Z_{k-1}(M),V/\L) \; \left| \; \parbox{2.3in}{$\exists \> \omega \in \Omega^k(M;V) $ satisfying \\
 $ \la \chi, \partial z\ra = \int_z \omega \mod \L \quad \forall \> z \in C_k(M)$ } \right. \right\}. \end{align*}
\end{defn}

\begin{prop}[\cite{CS85}]\label{prop:DiffCharProps}The groups $\H{k}{V}{\L}$ fit into the commutative square
\begin{equation}\label{eq:FundSquareV} \xymatrix{  \H{k}{V}{\L} \ar[d] \ar[r] & H^k(M;\L) \ar[d]\\
\Omega^k(M;V)_\L \ar[r] &H^k(M;V)_\L ,} \end{equation}
and these homomorphisms determine the short exact sequences
\begin{align} 
\label{eq:exact1} 0 \to H^{k-1}(M;V/\L) \to &\H[M]{k}{V}{\L} \to \Omega^k(M;V)_\L \to 0,  \\
\label{eq:exact2} 0 \to \frac{\Omega^{k-1}(M;V)}{\Omega^{k-1}(M;V)_\L} \to &\H{k}{V}{\L} \to H^k(M;\L) \to 0, \\
\label{eq:exact3} 0 \to \frac{H^{k-1}(M;V)}{H^{k-1}(M;V)_\L} \to &\H{k}{V}{\L} \to A^k(M;\L) \to 0. \end{align}
Here, $A^k (M;\L)$ is defined to be the the pullback in diagram \eqref{eq:FundSquareV}.  The induced homomorphism $H^{k-1}(M;V/\L)\to H^k(M;\L)$, given by \eqref{eq:exact1} and \eqref{eq:exact2}, equals minus the Bockstein homomorphism.  The homomorphism $\Omega^{k-1}(M;V)/\Omega^{k-1}(M;V)_\L \to \Omega^k(M;V)_\L$, induced from \eqref{eq:exact2} and \eqref{eq:exact1}, equals the de Rham differential $d$.
\end{prop}

\begin{rem}The notation $\H{k}{V}{\L}$ may seem a bit confusing, since the definition involves both $V$ and $\L$.  However, we view $\L$ not as an abstract group, but as a subgroup of a vector space $V$.
\end{rem}

\begin{rem}There is no real loss in the reader replacing $V$ by $\R$ throughout this paper.  While $V=\R$ and $\L=\Z$ in most examples and in much of the cited literature, the definitions and properties easily generalize.  The key property used is that when $\L$ is totally disconnected, the natural homomorphism
\[ \Omega^k(M;V) \to C^k(M;V/\L), \]
given by integrating and quotienting modulo $\L$, is injective.  This fact can be easily proven pointwise by integrating a $k$-form over an arbitrarily small disk.
\end{rem}

\begin{rem}One can also replace $\L \subset V$ by a totally disconnected subgroup of a $\Z$-graded vector space $\cL \subset \cV$.  This is useful when considering ``Chern characters" in differential generalized cohomology.  For example, the ordinary Chern character in $K$-theory can be refined to a natural transformation
\[ \check{K}^*(M) \overset{\check{\operatorname{ch}}}\longrightarrow \H{*}{\R[\beta^\pm]}{\Q[\beta^\pm]},\]
where $\Q[\beta^\pm] \subset \R[\beta^\pm]$ and $| \beta | = 2$ (\cite{BS09}).
\end{rem}

The homomorphism $\H{k}{V}{\L} \to \Omega^k(M;V)$ is given by the form $\omega$ in Definition \ref{defn:diffchars}.  We call this form the {\it curvature} of $\chi$, and we refer to the class in $H^k(M;\L)$ as the {\it characteristic class} of $\chi$.  Diagram \ref{eq:FundSquareV} implies that the curvature equals the characteristic class in de Rham cohomology.  We refer to the value $\la \chi, z\ra$ as the {\it holonomy} of $\chi$ along $z \in C_{k-1}(M)$.  This terminology is made particularly clear in the second example below.

\begin{example}\label{ex:functions}There is a canonical isomorphism
\[ C^\infty(M, \R/\Z) \iso \H{1}{\R}{\Z} \subset \Hom(Z_0(M), \R/\Z) .\]
To a differential character given by a function $f \in C^\infty(M, \R/\Z)$, the characteristic class is the homotopy class
\[ [f] \in [M, \R/\Z] \iso [M, K(\Z,1)] \iso H^1(M;\Z),\]
and the curvature is the derivative $df \in \Omega^1(M)$.
\end{example}

\begin{example}\label{ex:linebundles}Let $(P\overset{\pi}\to M, \Theta)$ be a principal $\R/\Z$-bundle with connection $\Theta$.  The holonomy of $\Theta$ associates an element in $\R/\Z$ to any closed loop in $M$.  The holonomy is invariant under gauge transformations, inducing an isomorphism
\[ \left\{ \R/\Z \into (P,\Theta) \to M\right\}_{/\sim}\; \overset{\iso}\longrightarrow \;\H{2}{\R}{\Z}. \]
The characteristic class of a degree 2 character is the cohomology class which classifies $P$, and the character's curvature equals the curvature $d\Theta$ of the connection.  More generally, $\H{2}{V}{\L}$ classifies principal $V/\L$-bundles with connection when $V/\L$ is a Lie group.  This fact is used in the discussion of $T$-duality in \cite{KahleVal09}.
\end{example}

\begin{rem}The Lie group isomorphism 
\[ \R/\Z \overset{\exp(2\pi i \cdot)}\longrightarrow U(1)\]
gives an isomorphism between $\H{2}{\R}{\Z}$ and gauge equivalence classes of principal $U(1)$-bundles with connection on $M$.  Under this isomorphism, the characteristic class equals minus the first Chern class, and the curvature equals minus the first Chern form.  Because of this sign, and because the group structures in $\H{1}{\R}{\Z}$ and $\H{2}{\R}{\Z}$ are naturally written additively, it will be easier to work with principal $\R/\Z$-bundles.
\end{rem}

While equating two bundles that are gauge equivalent makes many things simpler, one loses important information.  For example, local cutting and gluing constructions can not be performed on isomorphism classes.  This is due to the existence of non-trivial automorphisms.  In general, one should consider the category of principal $\R/\Z$-bundles with connection on $M$, where the morphisms are bundle maps preserving connection.

For the same reasons, it is useful to consider a category of differential cocycles $\cH{k}{V}{\L}$ whose set of equivalence classes $\pi_0\left( \cH{k}{V}{\L} \right)$ is canonically isomorphic to the groups $\H{k}{V}{\L}$.  We choose to use the following cochain model, constructed by Hopkins--Singer.  The category $\cH{2}{\R}{\Z}$ is equivalent to the category of $\R/\Z$-bundles with connection on $M$, as loosely explained in Section \ref{sec:Examples}.

\begin{defn}[\cite{HS05}]\label{defn:diffcomplex}The {\it differential cochain complex} $\{\Cbg{n}{*}{V}{\L}, d\}$ is given by
\[
\Cbg{n}{k}{V}{\L} = \begin{cases} C^k(M;\L) \times C^{k-1}(M;V) \times \Omega^k(M;V) & k \geq n \\
C^k(M;\L) \times C^{k-1}(M;V) & k < n \end{cases}  \]
with differential
\begin{align*}
d(c,h,\omega) &= (\delta c, \omega - c - \delta h, d\omega) \\
d(c,h) &= \begin{cases} (\delta c, -c-\delta h, 0)  & (c,h)\in \Cbg{k}{k-1}{V}{\L} \\
 (\delta c, - c - \delta h) &\text{otherwise.} \end{cases}
\end{align*}
{\it Differential cocycles} are cocycles in this complex, and $\Hbg{n}{k}{V}{\L}$ is the $k$-th cohomology; i.e.
\begin{align*} \Zbg{n}{k}{V}{\L} &= \{  \x \in \Cbg{n}{k}{V}{\L} \st d\x = 0 \}, \\
\Hbg{n}{k}{V}{\L} &= \Zbg{n}{k}{V}{\L}/ d\Cbg{n}{k-1}{V}{\L}.
\end{align*}
\end{defn}

One can check that for $n \neq k$, the cohomology groups of the differential cochain complex are naturally isomorphic to ordinary cohomology groups.  When $n=k$, we obtain the differential cohomology groups:
\begin{equation}\label{eq:BiGradedCoh}
\Hbg{n}{k}{V}{\L} \iso \begin{cases}H^k(M;\L) &k >n \\
\H{k}{V}{\L} &k=n\\
H^{k-1}(M;V/\L) & k < n.
\end{cases}
\end{equation}

\begin{defn}[\cite{HS05}]The category of differential $k$-cocycles $\cH{k}{V}{\L}$ is the fundamental groupoid
\[ \cH{k}{V}{\L} \= \pi_{\leq 1} \left(  \Cbg{k}{k-2}{V}{\L} \overset{d}\to \Cbg{k}{k-1}{V}{\L} \overset{d}\to \Zbg{k}{k}{V}{\L}   \right). \]
In terms of objects and morphisms, this means
\begin{align*} \Ob \cH{k}{V}{\L} &=  \Zbg{k}{k}{V}{\L} , \\
\Hom (\x_1, \x_2) &= \{ \check y \in \Cbg{k}{k-1}{V}{\L} \st d \check y = \x_2 - \x_1 \} / d \Cbg{k}{k-2}{V}{\L} \\
& \quad \quad \subset \Cbg{k}{k-1}{V}{\L} / d\Cbg{k}{k-2}{V}{\L}. \end{align*}
\end{defn}

The category $\cH{k}{V}{\L}$ has
\[ \pi_0 \left( \cH{k}{V}{\L} \right) \iso \H{k}{V}{\L},\]
and we denote an object $\x = (c,h,\omega)\in \Zbg{k}{k}{V}{\L}$ by
\[ \x  \in \cH{k}{V}{\L}.\]
From the triple $\x=(c,h,\omega)$, $c$ determines the characteristic class, $h$ determines the holonomy, and $\omega$ is the curvature.

To see this explicitly, a differential cocycle $\x$ is a triple
\[ \x = (c, h, \omega) \in C^k(M;\L) \times C^{k-1}(M;V) \times \Omega^k(M;V) \]
satisfying
\[ \delta c =0, \qquad \delta h = \omega - c, \qquad d\omega = 0.\]
In other words, $c$ and $\omega$ are closed, and the choice of $h$ specifies the equivalence $[c]=[\omega] \in H^k(M;V)$ at the cochain level.  One easily checks that the standard commutative square \eqref{eq:FundSquareV} is induced by the obvious maps below.
\[ \vcenter{ \xymatrix{ [(c,h,\omega)] \ar@{|->}[r] \ar@{|->}[d]& [c] \ar@{|->}[d] \\
\omega \ar@{|->}[r] &[\omega]=[c] } }
\quad \Rightarrow \quad 
\vcenter{ \xymatrix{ \Hbg{k}{k}{V}{\L} \ar[r] \ar[d]& H^k(M;\L) \ar[d] \\
\Omega^k(M;V)_\L \ar[r] &H^k(M;V)_\L
}} \]
Furthermore, the canonical isomorphism from $\Hbg{k}{k}{V}{\L}$ to the differential characters $\H{k}{V}{\L}$ is given by
\begin{align*}
\Hbg{k}{k}{V}{\L} &\overset{\iso}\longrightarrow \H{k}{V}{\L} \\
[(c,h,\omega)] &\longmapsto \la h, \bullet \ra \mod \L.
\end{align*}
In the future, we will not distinguish between $\Hbg{k}{k}{V}{\L}$ and $\H{k}{V}{\L}$.

Because the group $H^k(M;V)_\L$ appears frequently, it is worthwhile noting the following.

\begin{lemma}\label{lem:uct}There is a natural isomorphism \[H^k(M;V)_\L \iso \Hom(H_k(M),\L).\]
\begin{proof}
Because $\Hom(H_k(M), -)$ is a left-exact functor, the inclusion $\L \into V$ naturally induces an injective homomorphism
\[ \Hom(H_k(M), \L) \into \Hom(H_k(M), V).\]
That $V$ is torsion-free implies $\Ext(-,V)=0$.  Combining these facts with the universal coefficient theorem gives the following.
\[ \xymatrix@R=5mm{ 0 \ar[r]& \Ext(H_{k-1}(M), \L) \ar[d] \ar[r]& H^k(M;\L) \ar[r] \ar[d]& \Hom(H_k(M), \L) \ar[r] \ar@{^(->}[d]& 0 \\
0 \ar[r]& 0 \ar[r]& H^k(M;V) \ar[r]^<<<<<{\iso}& \Hom(H_k(M), V) \ar[r]& 0. 
}\]
Therefore, the image of $H^k(M;\L) \to H^k(M;V)\iso \Hom(H_k(M),V)$ is equal to the image of $\Hom(H_k(M),\L)$.
\end{proof} \end{lemma}

\renewcommand{\H}[3]{\check{H}^{#1}(#3)}
\subsection{Change of coefficients}
For spacing purposes, we use the notation $\H{k}{V}{\L}$ in this subsection to refer to the differential cohomology functor.  In other words, $\H{k}{V}{\L}$ denotes the groups $\check{H}^k(M;\L)$ that are functorially associated to $M$.

Let $0 \to \L_1 \overset{i}\into \L_2 \to V$ be an inclusion of totally disconnected subgroups of $V$.  This induces homomorphisms
\begin{equation}\label{eq:ChangeOfCoeff} \H{k}{V}{\L_1} \overset{\check{i}}\longrightarrow \H{k}{V}{\L_2}. \end{equation}
The short exact sequences
\[ 0 \to \L_1 \overset{i}\to \L_2 \to \L_2/ \L_1 \to 0  \quad \text{ and } \quad 0 \to \L_2 / \L_1 \to V/\L_1 \to V/\L_2 \to 0\]
also give rise to the long exact sequences
\begin{align*}
 \cdots \to H^k(\L_1) \to H^k(\L_2) \to &H^k(\L_2/\L_1) \overset{\beta}\to H^{k+1}(\L_1) \to \cdots, \\
 \cdots \to H^k(\L_2/\L_2) \to H^k(V/\L_1) \overset{i_*}\to &H^k(V/\L_2) \overset{\beta}\to H^{k+1}(\L_2/\L_1) \to \cdots ,\end{align*}
where $\beta$ is a Bockstein homomorphism.

\begin{prop}\label{prop:latticechange}The inclusion $\L_1 \overset{i}\into \L_2 \to V$ induces exact sequences
\begin{align*} \cdots \overset{i_*}\to H^{k-2}(V/\L_2) \overset{\beta}\to H^{k-1}(\L_2 / \L_1) \to &\H{k}{V}{\L_1} \overset{\check{i}}\to \H{k}{V}{\L_2} \to H^k(\L_2/\L_1) \overset{\beta} \to H^{k+1}(\L_1) \to \cdots, \\
0 \to \Ker i_* \to &\H{k}{V}{\L_1} \overset{\check{i}}\to \H{k}{V}{\L_2} \to \frac{H^k(\L_2)} {H^k( \L_1)} \to 0, 
\end{align*}
where $i_*: H^{k-1}(V/\L_1) \to H^{k-1}(V/\L_2)$.
\end{prop}

The proof of Proposition \ref{prop:latticechange} will follow in a moment.  It is based on the homological algebra Lemma \ref{lem:exactgrid}, whose proof is a diagram chase left to the reader.  Before giving these details, we wish to highlight the case where $\L_1 = 0$ and $\L_2 = \L$.  From the long exact sequence
\[\cdots \to H^{k-1}(\L) \to H^{k-1}(V) \overset{i_*}\to H^{k-1}(V/\L) \to \cdots, \]
we obtain $\Ker i_* = H^{k-1}(V)_\L$.  Plugging into Proposition \ref{prop:latticechange} gives the important exact sequence
\begin{equation}\label{eq:ses0Z} 0\to  H^{k-1}(V)_\L \to  \H{k}{V}{0} \to \H{k}{V}{\L} \to H^k(\L) \to 0.
\end{equation}

\begin{rem}In \cite{CS85}, it is claimed that
\[ 0 \to \Ker i_* \to \H{k}{\R}{\L_1} \to \H{k}{\R}{\L_2} \to \Omega_{\L_2}^k/ \Omega_{\L_1}^k \to 0 \]
is exact, but Proposition \ref{prop:latticechange} shows this is not correct.  While there is a surjective homomorphism $H^k(\L_2)/H^k(\L_1) \to \Omega_{\L_2}^k/ \Omega_{\L_1}^k$, it can fail to be injective.  When $\L_1=0$ and $\L_2=\Z$, the relevant map $H^k(\Z) \to \Omega_\Z^k / \Omega^k_0 \iso H^k(\R)_\Z$ is an isomorphism if and only if $H^k(\Z)$ is torsion-free.
\end{rem}

\begin{lemma}\label{lem:exactgrid}Assume the following vertical and horizontal sequences are exact.
\[ \xymatrix@=6mm{ 
&\vdots \ar[d] \\
&A_0 \ar[d]^{f}&&0\ar[d] \\
0 \ar[r] &A_1\ar[d] \ar[r]^{g}&B_1\ar[d] \ar[r]&C_1 \ar[d]\ar[r]& 0 \\
0 \ar[r]&A_2\ar[r] &B_2 \ar[r] & C_2\ar[r] &0 \\
}  \hskip10mm \xymatrix@=6mm{
0 \ar[r] &A'_1\ar[d] \ar[r]&B'_1\ar[d] \ar[r]&C'_1 \ar[d]\ar[r]& 0 \\
0 \ar[r]&A'_2\ar[r] \ar[d]&B'_2 \ar[r]^{g'} & C'_2\ar[d]^{f'} \ar[r] &0 \\
&0 &&C'_3\ar[d] \\
&&&\vdots
} \]
Then, the sequences
\[ \cdots \to A_{-1} \to A_0 \overset{g\o f}\longrightarrow B_1 \to B_2 \text{ and } B'_1 \to B'_2 \overset{f' \o g'}\longrightarrow C_3' \to C_4' \to \cdots \]
are exact.
\begin{proof}The proof is a standard diagram chase and is left to the reader.
\end{proof}
\end{lemma}

\begin{proof}[Proof of Proposition \ref{prop:latticechange}]
The proof is a direct application of Lemma \ref{lem:exactgrid}.  The short exact sequence \eqref{eq:exact1} produces the following exact sequences.
\[ \xymatrix@R=4mm{ 
&\vdots \ar[d]^{\beta} \\
&H^{k-1}(\L_2/ \L_1) \ar[d]^{j_*}&&0\ar[d] \\
0 \ar[r] & H^{k-1}(V / \L_1) \ar[d]^{i_*} \ar[r]& \H{k}{V}{\L_1} \ar[d] \ar[r]& \Omega^k_{\L_1} \ar[d]\ar[r]& 0 \\
0 \ar[r]& H^{k-1}(V/ \L_2) \ar[r] \ar[d]^{\beta}& \H{k}{V}{\L_2} \ar[r] & \Omega^k_{\L_2} \ar[d]\ar[r] &0 \\
& H^{k}(\L_2/ \L_1) \ar[d] &&  \Omega^k_{\L_2}/ \Omega^k_{\L_1}\ar[d] \\
&\vdots &&0
} \]
The first case of Lemma \ref{lem:exactgrid} then gives the exact sequence
\[ H^{k-2}(V/\L_1) \overset{i_*}\to  H^{k-2}(V/\L_2) \overset{\beta}\to H^{k-1}(\L_2 / \L_1) \to \H{k}{V}{\L_1} \to \H{k}{V}{\L_2}.  \]
Using the isomorphisms
\[ \frac{H^{k-1}(\L_2/\L_1)}{ \beta (H^{k-2}(V/\L_2) )}\iso \frac{H^{k-1}(\L_2/\L_1)}{ \Ker j_* }\iso  \Image j_* \iso \Ker i_*,\]
where $i_*: H^{k-1}(V/\L_1) \to H^{k-1}(V/\L_2)$, we arrive at the exact sequence
\[ 0 \to \Ker i_* \to \H{k}{V}{\L_1} \to \H{k}{V}{\L_2}.\]
The short exact sequence \eqref{eq:exact2} produces the following exact sequences.
\[\xymatrix@R=4mm{
&0\ar[d] && \vdots \ar[d] \\
&\Omega^{k-1}_{\L_2} / { \Omega^{k-1}_{\L_1} } \ar[d]&& H^{k-1} (\L_2/\L_1)\ar[d]^{\beta} \\
0 \ar[r] & \Omega^{k-1} / \Omega^{k-1}_{\L_1}  \ar[d] \ar[r]& \H{k}{V}{\L_1} \ar[d] \ar[r]& H^k(\L_1) \ar[d]\ar[r]& 0 \\
0 \ar[r]&  \Omega^{k-1} / \Omega^{k-1}_{\L_2} \ar[r] \ar[d]& \H{k}{V}{\L_2} \ar[r] & H^k(\L_2) \ar[d]\ar[r] &0 \\
&0 && H^k(\L_2/\L_1)\ar[d]^{\beta} \\
&&&\vdots
}\]
The second case of Lemma \ref{lem:exactgrid} then gives the exact sequence
\[ \H{k}{V}{\L_1} \to \H{k}{V}{\L_2} \to H^k(\L_2 / \L_1) \overset{\beta}\to H^{k+1}(\L_1) \to \cdots. \]
Using the injective map $\frac{H^k(\L_2)} {H^k( \L_1)} \to H^k(\L_2/\L_1)$, we can rewrite this
\[ \H{k}{V}{\L_1} \to \H{k}{V}{\L_2} \to \frac{H^k(\L_2)} {H^k( \L_1)} \to 0.\qedhere\]
\end{proof}


\section{Trivializations in differential ordinary cohomology}\label{sec:TrivOrdinary}
\renewcommand{\H}[4][M]{\check{H}^{#2}(#1; #4)}

We now discuss trivializations of cocycles in differential ordinary cohomology.  Again, this idea is not new, and our treatment is heavily influenced by the papers \cite{HS05} and \cite{Fre00}.  The goal is simply to generalize the fundamental commutative square \eqref{eq:FundSquareV} to trivializations.

\subsection{Two settings for trivializations}
In general, one can define trivializations whenever there is some notion of objects, morphisms, and a specified identity object.  A trivialization is then a choice of isomorphism between an object and the identity object.  In this paper, though, we will have one of the following two situations in mind.

\begin{example}\label{ex:trivcat}Let $\{C^*, \delta\}$ be a cochain complex.  Associated to any cocycle $x \in Z^k$ is the category of trivializations of $x$, denoted $\mathcal{T}(x)$.  Succinctly, $\mathcal{T}(x)$ is the fundamental groupoid
\[ \mathcal{T}(x) \= \pi_{\leq 1} \left(  C^{k-3} \overset{\delta}\to C^{k-2} \overset{\delta}\to \delta^{-1}(x) \right),\] 
where $\delta^{-1}(x) \subset C^{k-1}$ is the pre-image of $x$ under $\delta$.  In other words, the objects are cochains $y \in C^{k-1}$ such that $\delta y = x$, and the morphisms are given by the additive action of $C^{k-2}/\delta C^{k-3}$.  Such a trivialization exists precisely when $x\in \delta C^{k-1}$, which is equivalent to saying $[x] = 0 \in H^n(C^*)$.  Assuming $[x]=0$, the set of objects is a torsor over $Z^{k-1}$.  Denoting the set of isomorphism classes by
\[ T(x) \= \pi_0 \left( \mathcal{T}(x) \right),\]
one easily sees that $T(x)$ is a torsor for $Z^{k-1}/\delta C^{k-2} = H^{k-1}(C^*)$.
\end{example}

\begin{example}\label{ex:trivset}Let $G_1 \overset{f}\to G_2$ be a group homomorphism.  Then, for $g_2 \in G_2$, we define the set of trivializations of $g_2$ as the inverse image of $g_2$ under $f$; i.e.
\[ T(g_2) \= f^{-1}(g_2) = \{ g \in G_1 \st f(g_1) = g_2\} \subset G_1.\]
If $g_2 \in \Image f$, then $T(g_2)$ is a torsor for $\Ker f$.
\end{example}

\subsection{Trivializations of differential cocycles}As in Section \ref{sec:Background}, assume $M$ is a smooth manifold, $\L \subset V$ is a totally disconnected subgroup of a vector space, and $\Cbg{*}{*}{V}{\L}$ is the bi-graded complex of differential cochains.

\begin{defn}\label{defn:trivgeo}
Let $\x = (c,h,\omega) \in \cH{k}{V}{\L}$.  The category of trivializations of the differential cocycle $\x$ is 
\[ \cT{\x} \= \pi_{\leq 1} \left(  \Cbg{k-1}{k-3}{V}{\L} \overset{d}\to \Cbg{k-1}{k-2}{V}{\L} \overset{d}\to d^{-1}(\x) \right),\]
where $d^{-1}(\x) \subset \Cbg{k-1}{k-1}{V}{\L}$.  The set of equivalence classes is denoted
\[ \T{\x} \= \pi_0\left( \cT{\x} \right).\]
\end{defn}

Let us briefly clarify the above definition.  There is a natural equality 
\[\Zbg{k}{k}{V}{\L} = \Zbg{k-1}{k}{V}{\L}.\]  If $\x \in \Zbg{k}{k}{V}{\L}$, we can consider the category of trivializations using either $\{ \Cbg{k}{*}{V}{\L} \}$ or $\{ \Cbg{k-1}{*}{V}{\L}\}$.  We choose the latter option.  To better visualize this, we draw the bi-graded differential cochain complex. The region above the dotted line contains no differential forms.
\[ \xymatrix@!0@R=6mm@C=22mm{ &\Cbg{k}{k-2}{V}{\L} \ar[rr]^{d} \ar[dd]^{=} && \Cbg{k}{k-1}{V}{\L} \ar[rr]^{d} \ar@{^(->}[dd]&&\Cbg{k}{k}{V}{\L} \ar[dd]^{=}\\ &&{\color{darkgray}\omega = 0} \\
&\Cbg{k-1}{k-2}{V}{\L} \ar[rr]^{d} \ar@{^(->}[dd] &&\Cbg{k-1}{k-1}{V}{\L} \ar[rr]^{d} \ar[dd]^{=}&&\Cbg{k-1}{k}{V}{\L} \ar[dd]^{=} \\ &&&&{\color{darkgray} \omega \in \Omega^*(M;V)} \\
\ar@{--}[uuuurrrr] &\Cbg{k-2}{k-2}{V}{\L} \ar[rr]^{d}&&\Cbg{k-2}{k-1}{V}{\L} \ar[rr]^{d} &&\Cbg{k-2}{k}{V}{\L} }
\]
A differential cocycle $\x \in \Zbg{k}{k}{V}{\L}$ naturally lives in the top row, but we define the trivializations using the second row.

Therefore, a trivialization of $\x$ exists if and only if $[\x] = 0 \in \Hbg{k-1}{k}{V}{\L} \iso H^k(M;\L)$.  In particular, a trivialization of $\x$ may exist when the form $\omega \neq 0$.  This is weaker than the condition $[\x] = 0 \in \H{k}{V}{\L}$, which is the obstruction to trivializing $\x$ in the first row $\Cbg{k}{*}{V}{\L}$.  

\begin{rem}The category of trivializations defined using the complex $\{\Cbg{k}{*}{V}{\L}\}$ is of relevance as well.  Often these are referred to as geometric trivializations.  In degree 2, this is the category which describes connection preserving isomorphisms of $\R/\Z$-bundles.  To avoid confusion, we will not call objects in $\cT{\x}$ geometric trivializations.  The subscript geo simply refers to the fact that the trivialization contains local geometric data and is non-canonically equivalent to differential character.  Similarly, the subscripts top, curv, and dR should convey that these structures generalize the characteristic class, curvature, and de Rham class of a differential character.
\end{rem}

\begin{defn}\label{defn:trivtop}Let $\x = (c,h,\omega) \in \cH{k}{V}{\L}$.  The trivializations of the characteristic cocycle $c$ form a category
\[ \cTt{\x} \= \mathcal{T}(c) = \pi_{\leq 1} \left( C^{k-3}(M;\L) \overset{\delta}\to C^{k-2}(M;\L) \overset{\delta}\to \delta^{-1}(c)  \right).\]
The equivalence classes are denoted
\[ \Tt{\x} \= \pi_0\left( \cTt{\x} \right).\]
\end{defn}

\begin{defn}\label{defn:trivcurv}Let $[\x] \in \H{k}{V}{\L}$.  The set $\Tc{\x} = \Tc{[\x]}$ is the pre-image of $[\x]$ in the homomorphism
\[ \Omega^{k-1}(M;V) \to \H{k}{V}{\L} \]
defined by \eqref{eq:exact2}.  In other words,
\[ \Tc{\x} \= \left\{ \eta \in \Omega^{k-1}(M;V) \st \int_\bullet \eta  = \la [\x], \bullet \ra \in \Hom(Z_{k-1}(M), V/\L) \right\},\]
and elements can be thought of as globally defined connection forms on $M$.
\end{defn}

\begin{defn}\label{defn:trivhol}The set $\Th{\x} = \Th{[\x]}$ is the pre-image of $[\x]$ in the homomorphism
\[ \H{k}{V}{0} \to \H{k}{V}{\L}\]
defined by the inclusion $0 \into \L$; i.e. elements are lifts of the character from $V/\L$ to $V$:
\[ \Th{\x} \= \left\{ \wt{\chi} \in \Hom(Z_{k-1}(M), V) \st \la \wt \chi, \bullet \ra = \la [\x], \bullet \ra \mod \L \right\}. \]
\end{defn}

\begin{rem}Plugging in $\L=0$ to the short exact sequence \eqref{eq:exact2} gives a canonical isomorphism
\begin{equation}\label{eq:IsoTopH0} \frac{\Omega^{k-1}(M;V)}{\Omega^{k-1}(M;V)_0} = \frac{\Omega^{k-1}(M;V)}{d\Omega^{k-2}(M;V)} \iso \H{k}{V}{0}.\end{equation}
\end{rem}

\begin{rem}We also consider the sets $\Tc{\x}$ and $\Th{\x}$ as categories with no morphims other than the identity.
\end{rem}

\begin{prop}\label{prop:Obstruction}For $\x =(c,h,\omega) \in \cH{k}{V}{\L}$, the sets $\T{\x}, \Tt{\x}, \Tc{\x},$ and $\Th{\x}$ are all non-empty if and only if the characteristic class $[c] = 0 \in H^k(M;\L)$.  If $[c]=0$, then these sets are torsors for $\H{k-1}{V}{\L}$, $H^{k-1}(M;\L)$, $\Omega^{k-1}(M;V)_\L$, and $H^{k-1}(M;V)_\L$ respectively.
\begin{proof}
The categories $\cT{\x}$ and $\cTt{\x}$ are defined using the general construction in Example \ref{ex:trivcat}.  The relevant cochain complexes are
\begin{align*}
 \cdots \overset{d}\to \Cbg{k-1}{k-2}{V}{\L} \overset{d}\to \Cbg{k-1}{k-1}{V}{\L} \overset{d}\to \Zbg{k-1}{k}{V}{\L}, \\ 
\cdots \overset{\delta}\to C^{k-2}(M;\L) \overset{\delta}\to C^{k-2}(M;\L) \overset{\delta}\to C^{k-1}(M;\L) \overset{\delta}\to Z^k(M;\L).
\end{align*}
Therefore, $\T{\x}$ is non-empty if and only if 
\[ [\x] = 0 \in \Hbg{k-1}{k}{V}{\L} \iso H^k(M;\L),\]
where we use the isomorphism \eqref{eq:BiGradedCoh}.  If $[c]=0$, then $\T{\x}$ is a torsor for 
\[ \Zbg{k-1}{k-1}{V}{\L}/d \Cbg{k-1}{k-2}{V}{\L} \iso \H{k-1}{V}{\L}.\]  The second cochain complex immediately implies $\Tt{\x}$ is non-empty if and only if $[c]=0\in H^k(M;\L)$.  If $[c]=0$, then $\Tt{\x}$ is a torsor for $H^{k-1}(M;\L)$.

The sets $\Tc{\x}$ and $\Th{\x}$ are defined using the general construction in Example \ref{ex:trivset}, and we rewrite the relevant exact sequences \eqref{eq:exact2} and \eqref{eq:ses0Z}:
\[ \xymatrix@C=5mm@R=4mm{0 \ar[r]& \Omega^{k-1}(M;V)_\L  \ar[r] \ar[dd]^{\big/ d\Omega^{k-2}} & \Omega^{k-1}(M;V) \ar[dr] \ar[dd]^{\big/ d\Omega^{k-2}}\\
 &&& \H{k}{V}{\L} \ar[r]& H^k(M;\L) \ar[r]& 0 \\
0 \ar[r]& H^{k-1}(M;V)_\L \ar[r]& \H{k}{V}{0} \ar[ur]
}\]
Therefore, $\Tc{\x}$ and $\Th{\x}$ are both non-empty precisely when the characteristic class $[c]=0 \in H^k(M;\L)$.  Assuming $[c]=0$, then $\Tc{\x}$ is a torsor for $\Omega^{k-1}(M;V)_\L$, and $\Th{\x}$ is a torsor for $H^{k-1}(M;V)_\L$.
\end{proof}\end{prop}

\begin{prop}\label{prop:trivsquare}Let $\x = (c,h,\omega) \in \cH{k}{V}{\L}$ with $[c]=0\in H^k(M;\L)$.  Then, we have commutative diagrams of functors and sets
\[ \xymatrix{ \cT{\x}  \ar[d] \ar[r]& \cTt[\l]{\x} \ar[d]\\
\Tc[\l]{\x} \ar[r] & \Th[\l]{\x}
} \qquad \qquad \xymatrix{ \T{\x}  \ar[d] \ar[r]& \Tt[\l]{\x} \ar[d]\\
\Tc[\l]{\x} \ar[r] & \Th[\l]{\x}
} \]
induced by 
\[ \xymatrix{ (b,k,\eta) \ar@{|->}[r] \ar@{|->}[d]& b \ar@{|->}[d]\\
\eta \ar@{|->}[r] & \int_{\bullet}\eta  \>=\>  \la h+b, \bullet \ra.}\]
\begin{proof}
We first establish the maps at the level of objects.  Let $\s = (b,k,\eta) \in C^{k-1}(M;\L) \times C^{k-2}(M;V) \times \Omega^{k-1}(M)$ be a trivialization of $\x$.  Expanding the relation $d \s = \x$, we see
\begin{align}
\label{eq:toptriv} \delta b &= c, \\ 
\label{eq:liftscoincide} \eta - b - \delta k &= h, \\ 
\nonumber d\eta &= \omega.
\end{align}
Equation \eqref{eq:toptriv} implies $b \in  \cTt{c}$.  

Let $z\in Z_{k-1}(M)$. The second equation \eqref{eq:liftscoincide}, along with the facts that $b(z) \in \L$ and $\delta k(z) = k(\d z) =0$, implies
\begin{align*}\label{eq:liftscoincide} \int_z \eta &= \la h + b + \delta k, z\ra \\
 &= \la h, z\ra \mod \L.
\end{align*}
Hence $\eta$ is a global connection form for $\x$; i.e. $\eta \in \Tc{\x}$.

Given any $\eta \in \Tc{\x}$, we define the class in $\Th{\x} \subset \H{k}{V}{0}$ using the standard map $\Omega^{k-1}(M) \to \H{k}{V}{0}$ given by integration on closed cycles.  The maps
\[ \xymatrix{
\Omega^{k-1}(M)  \ar[r] \ar[drr] & \frac{\Omega^{k-1}(M;V)}{d\Omega^{k-2}(M;V)} \ar[r]^{\iso}& \H{k}{V}{0} \ar[d] \\
&&\H{k}{V}{\L}  }  \]
obviously commute, and $\eta \mapsto [\x] \in \H{k}{V}{\L}$.  Therefore, the image of $\eta$ in $\H{k}{V}{0}$ is an element in $\Th{\x}$; integrating $\eta$ lifts of the holonomy of $\x$ from $V/\L$ to $V$.

Given $b \in  \cTt{\x} \subset C^{k-1}(M;\L)$, define the character $\wt \chi \in \H{k}{V}{0}$ by
\[ \la \wt \chi, \bullet \ra  \= \la h + b, \bullet\ra. \]
Since $\la \chi, \bullet\ra = \la h, \bullet \ra$ mod $\L$, then $\wt \chi \in \Th{\x}$.  The commutativity of the square follows from equation \eqref{eq:liftscoincide}.  If $z\in Z_{k-1}(M)$, then
\[ \int_z  \eta = \la h + b + \delta k, z\ra = \la h+b, z\ra. \]

We now define and check the functor on morphisms.  Let $\s_1 = (b,k,\eta) \in \cT{\x}$, and let $(a,l) \in C^{k-2}(M;\L) \times C^{k-3}(M;V)$ determine a morphism in $\Hom_{\cT{\x}}(\s_1, \s_2)$ by $\s_2 = \s_1 + d(a,l).$  We then see that
\[ \s_2 = (b,k,\eta) + d(a,l) = (b + \delta a, k - a - \delta l,\eta).\]
The differential form component of $\s_2$ still equals $\eta$, so the map from $\cT{\x} \to \Tc{\x}$ is well defined.

The functor $\cT{\x} \to \cTt{\x}$ is induced by the obvious map of cochain complexes; the morphism $(a,l)$ is sent to $a \in \Hom_{\Tt{\x}}(b, b+\delta a)$.  Finally, the morphism $a$ in $\cTt{\x}$ does not change the resulting class in $\Th{\x}$ since
\[ \la h + b + \delta a, \bullet \ra = \la h + b, \bullet \ra \in \Hom\left( Z_{k-1}(M), V \right).\]
The functor $\cTt{\x} \to \Th{\x}$ is therefore well-defined.

This completes the proof for the commutative diagram of functors.  Taking isomorphism classes yields the commutative diagram of sets.
\end{proof}\end{prop}

\begin{thm}\label{thm:torsorsquare}Let $\x = (c,h,\omega) \in \cH{k}{V}{\L}$ with $[c]=0 \in H^k(M;\L)$.  Then, the commutative square on the left is a torsor for the standard commutative square on the right.
\[ \xymatrix{ \T{\x}  \ar[d] \ar[r]& \Tt[\l]{\x} \ar[d]\\
\Tc[\l]{\x} \ar[r] & \Th[\l]{\x}
} \qquad \qquad
 \xymatrix{ \H{k-1}{V}{\L} \ar[d]\ar[r]&H^{k-1}(M;\L) \ar[d]\\
\Omega^{k-1}(M;V)_\L \ar[r] &H^{k-1}(M;V)_\L }
\]
\begin{proof}
Consider the commutative square of functors from Proposition \ref{prop:trivsquare} at the level of objects, which we draw below on the left.  By Proposition \ref{prop:Obstruction}, it is clear that each space on the left is a torsor for the corresponding group in the square on the right.
\[ \xymatrix{ \Ob \cT{\x}  \ar[d] \ar[r]& \Ob \cTt[\l]{\x} \ar[d]\\
\Tc[\l]{\x} \ar[r] & \Th[\l]{\x}
} \qquad 
 \xymatrix{ \Ob \cH{k-1}{V}{\L} \ar[d]\ar[r]& Z^{k-1}(M;\L) \ar[d]\\
\Omega^{k-1}(M;V)_\L \ar[r] &H^{k-1}(M;V)_\L }
\]
We check that the action is equivariant at the level objects by drawing the action of a cocycle $(b', k', \eta') \in   \cH{k-1}{V}{\L}$ on $(b,k,\eta) \in  \cT{\x}$ in the following diagram.
\[ \vcenter{\xymatrix@C=3mm{ (b+b', k+k', \eta+\eta') \ar@{|->}[d] \ar@{|->}[r] & b+b' \ar@{|->}[d] \\
\eta + \eta ' \ar@{|->}[r]& \int_{\bullet} \eta+ \eta'  \> = \> \big\la h+b+b', \bullet \big\ra}}
 \longmapsfrom 
\vcenter{\xymatrix@C=3mm{ (b', k', \eta') \ar@{|->}[d] \ar@{|->}[r]& b' \ar@{|->}[d] \\
\eta ' \ar@{|->}[r]& \int_\bullet \eta' = \la b', \bullet \ra }}  \]

The desired result follows by quotienting $\Ob \cT{\x}$ and $\Ob \cH{k-1}{V}{\L}$ by $d\Cbg{k-1}{k-2}{V}{\L}$, and  quotienting $\Ob \cTt{\x}$ and $Z^{k-1}(M;\L)$ by $dC^{k-2}(M;\L)$.\end{proof}\end{thm}

\begin{rem}Note that the categories $\cT{\x}$ and $\cH{k-1}{V}{\L}$ are non-canonically isomorphic.  The choice of an object in $\cT{\x}$ determines such an isomorphism.  In this sense, we can think of the category $\cT{\x}$ as being a torsor for the category $\cH{k-1}{V}{\L}$.  More precisely, the objects in $\cH{k-1}{V}{\L}$ form an abelian group which acts freely and transitively on the objects in $\cT{\x}$, and both sets of morphisms are canonically equal to $\Cbg{k-1}{k-2}{V}{\L}/d\Cbg{k-1}{k-3}{V}{\L}$.  The same statements hold for $\cTt{\x}$ and the category of singular cocycles $\mathcal{H}^{k-1}(M;\L)$.
\end{rem}

Theorem \ref{thm:torsorsquare} can be quite useful in practice.  First, all maps are surjective.  Second, the lack of injectivity is described explicitly by standard short exact sequences.  For example, when $H^{k-2}(M;V/\L)=0$, there is a natural isomorphism $\H{k-1}{V}{\L} \iso \Omega^{k-1}(M;V)_\L$.  In this case, the equivariance in Theorem \ref{thm:torsorsquare} implies that the natural map $\T{\x} \to \Tc{\x}$ is a bijection, so a trivialization of $\x$ is completely determined (up to isomorphism) by a differential form.

We now discuss the dependance of the categories $\cT{\x}$ and $\cTt{\x}$ on the choice of cocycle $\x$.  In examples, one may have a canonical differential cohomology class, but not a canonical cocycle.  For example, let $G$ be a compact Lie group with $\lambda \in H^{2k}(BG;\Z)$.  Associated to any principal $G$-bundle with connection $(P,\Theta) \overset{\pi}\to M$ is a canonical differential character $\l(P,\Theta) \in \H{2k}{\R}{\Z}$, but there is no canonical differential cocycle representing $\l(P,\Theta)$.  In the Hopkins--Singer model, for example, one usually assumes a classifying map for $P$ and a universal cocycle for $\lambda$.

\begin{prop}\label{prop:IsoCocycles}Let $\x$ and $\x + d \check y$ be isomorphic objects in $\cH{k}{V}{\L}$.  Then, the morphism $\check y = (b',k',0) \in \Cbg{k}{k-1}{V}{\L}$ determines isomorphisms $\cT{\x} \to \cT{\x + d\check y}$ and $\cTt{\x} \to \cTt{\x+d \check y}$ compatible with the maps in Proposition \ref{prop:trivsquare} and Theorem \ref{thm:torsorsquare}.
\begin{proof}The functors are defined in the obvious way
\[ \xymatrix@R=1mm{ &\cT{\x} \ar[r]  & \cT{\x + d \check y} &&\cTt{\x} \ar[r]&\cTt{\x + d\check y}\\
\Ob: &\s  \ar@{|->}[r]& \s + \check y &&b \ar@{|->}[r] &b + b' \\
\Mor: & \check t \ar@{|->}[r]& \check t  && a \ar@{|->}[r] & a
} \]
It is straightforward to check compatibility with the characteristic cocycle, curvature, and de Rham maps.
\end{proof}\end{prop}

In one sense, Proposition \ref{prop:IsoCocycles} says we can discuss trivializations of a differential cohomology class by simply picking a representing cocycle $\x \in \cH{k}{V}{\L}$.  A different choice $\x'$ will give rise to an isomorphic category of trivializations.  However, the isomorphism is not canonical!  It depends on $\y \in \Cbg{k}{k-1}{V}{\L}$, as opposed to only $d\y$.  The automorphisms
\[ \Aut(\x) = \frac{ \Zbg{k}{k-1}{V}{\L} }{ d \Cbg{k}{k-2}{V}{\L}} = \Hbg{k}{k-1}{V}{\L} \iso H^{k-2}(M;V/\L) \]
will induce automorphisms of the resulting category $\cT{\x}$.  One must be careful when performing cutting or pasting constructions, or when considering families of differential cohomology classes.  On the other hand, these subtleties can be ignored when using the curvature and de Rham classes because the sets $\Tc{\x}$ and $\Th{\x}$ only depend on the equivalence class $[\x] \in \H{k}{V}{\L}$.


\subsection{Hodge decomposition}\label{subsec:Hodge}
We now assume that $M$ is a closed (compact with no boundary) manifold with Riemannian metric $g$.  The Hodge Laplacian 
\[ \Delta_g^k  = dd^* + d^*d:\Omega^k(M)\to \Omega^k(M)\]
gives rise to the Hodge orthogonal decomposition of forms
\[ \Omega^k (M) = \Ker \Delta_g^k \oplus d\Omega^{k-1}(M) \oplus d^* \Omega^{k+1}(M),\]
and the canonical inclusion
\[ H^k(M;\R) \iso \Ker \Delta^k_g \into \Omega^k(M)\]
splits the short exact sequence
\[ \xymatrix{ 0 \ar[r]& d\Omega^{k-1}(M) \ar[r] & \Omega^k(M) \ar[r]& H^k(M;\R) \ar@{..>}@/_1pc/[l]_{Hodge}  \ar[r]& 0. } \]
If $V$ is a vector space, the Hodge decomposition extends to $\Omega^k(M;V)=\Omega^k(M)\otimes V $.

\begin{prop}\label{prop:Hodge}Let $(M,g)$ be a closed Riemannian manifold.  The Hodge decomposition of forms gives a canonical right inverse to the surjective map $\Tc{\x} \to \Th{\x}$, and it is equivariant with respect to the standard Hodge isomorphism $H^{k-1}(M;\R) \overset{\iso}\to \Ker \Delta_g^{k-1} \subset \Omega^{k-1}(M)$.
\[ \xymatrix{ \Tc[\l]{\x} \ar[r] & \Th[\l]{\x} \ar@{..>}@/_1pc/[l]_{Hodge}  && \Omega^{k-1}(M;V)_\L \ar[r] &H^{k-1}(M;V)_\L \ar@{..>}@/_1pc/[l]_{Hodge}} \]
\begin{proof}The sets $\Tc{\x}$ and $\Th{\x}$ are subsets of $\Omega^{k-1}(M;V)$ and $\H{k}{V}{0}$, respectively, so it suffices to construct a right inverse to $\Omega^{k-1}(M;V) \to \H{k}{V}{0}$.  As mentioned in \eqref{eq:IsoTopH0}, the short exact sequence \eqref{eq:exact2} gives an isomorphism
\[ \frac{\Omega^{k-1}(M;V)}{d\Omega^{k-2}(M;V)} \overset{\iso}\to \H{k}{V}{0}.\]
The Hodge decomposition gives a natural inclusion
\begin{align*} \H{k}{V}{0} &\iso \left( \Omega^{k-1}(M) \ominus d\Omega^{k-2}(M) \right) \otimes V \\&= \left( \Ker \Delta_g^{k-1} \oplus d^*\Omega^{k}(M)\right)\otimes V \subset \Omega^{k-1}(M;V). \end{align*}
This inclusion maps into the subspace of forms orthogonal to $d\Omega^{k-2}(M;V)$, so the composition
\[ \H{k}{V}{0} \into \Omega^{k-1}(M;V) \overset{/d\Omega^{k-2}}\longrightarrow \H{k}{V}{0} \]
is the identity.  Under this inclusion, classes in $\H{k}{V}{0}$ represented by closed forms are sent to their harmonic representative.
\end{proof}\end{prop}

In other words, to any $\wt {\chi} \in \Th{\x}$, there is a unique form $\eta \in \Omega^{k-1}(M;V)$ satisfying both
\[ \int_\bullet \eta = \la \wt{\chi}, \bullet \ra \in \Hom(Z_{k-1}(M), V) \quad \text{and} \quad d^*\eta = 0. \]

Combined with Theorem \ref{thm:torsorsquare}, this says that when $(M,g)$ is closed Riemannian, there are natural compatible right inverses given by the dotted Hodge arrow in the following commutative squares.
\[ \xymatrix{ \T{\x}  \ar[d] \ar[r]& \Tt[\l]{\x} \ar[d]  && \H{k-1}{V}{\L} \ar[d]\ar[r]&H^{k-1}(M;\L) \ar[d] \\
\Tc[\l]{\x} \ar[r] & \Th[\l]{\x} \ar@{..>}@/_1pc/[l]_{Hodge}  && \Omega^{k-1}(M;V)_\L \ar[r] &H^{k-1}(M;V)_\L \ar@{..>}@/_1pc/[l]_{Hodge}} \]
This give an isomorphism of the pullback in the above left square
\[ \Tt{\x} \times_{\Th{\x}} \Tc{\x} \iso \Tt{\x} \times d\Omega^{k-2}(M;V).\]
Differential characters with harmonic curvature were considered in \cite{MR2518994}, where they were referred to as harmonic Cheeger--Simons characters.  Continuing this analogy, one can define harmonic trivializations as trivializations whose curvature $\eta$ satisfies $d^*\eta=0$.  For a degree $k$ differential cocycle, the set (of isomorphism classes) of harmonic trivializations is a torsor for the harmonic characters of degree $k-1$.

A natural question is whether the Hodge decomposition can be refined to give canonical right inverses $\H{k}{V}{\L} \from H^k(M;\L)$ or $\T{\x} \from \Tt{\x}$.  It is worthwhile to point out that no such construction can exist, as we now explain.  We restrict to the category of smooth closed Riemannian manifolds, where the morphisms are isometries (i.e. diffeomorphisms compatible with the Riemannian metric).

\begin{prop}\label{prop:NoSplittings}If $V\neq 0$, any right splitting of the short exact sequences \eqref{eq:exact1}, \eqref{eq:exact2}, or \eqref{eq:exact3} will not be functorial in the category of closed Riemannian manifolds.
\begin{proof}
To prove the non-existence, we construct an explicit counter-example.  Since any $V$ admits an inclusion $\R \into V$, we only need to check the case $V=\R$.   

For the short exact sequence \eqref{eq:exact1}, assume that we have some some construction $\Omega^k(-)_\L \to \H[-]{k}{V}{\L}$, denoted by $\omega \mapsto \chi_{\omega}$, that splits \eqref{eq:exact1}.  Let ${\mathbb T}^k = \R^k/\Z^k$ be the $k$-dimensional torus with the flat metric induced from the Euclidean metric on $\R^k$, and with the standard coordinates $\{x^i\}$ inherited from $\R^k$.  Let $\omega = dx^1 \wedge \ldots \wedge dx^k$ be the volume form (which is also harmonic), and let $f_a:{\mathbb T}^k \to {\mathbb T}^k$ be an isometry given by translation in the first coordinate; i.e. $f(x^1,\ldots, x^k) = (x^1 +a, x^2, \ldots, x^k)$ where $a \in \R$.  Note that $f_a^* \omega = \omega$.  To compare $f_a^* \chi_\omega$ with $\chi_\omega$, we evaluate on the $(k-1)$-cycle $0\times \mathbb T^{k-1}$:
\begin{align*}
\la f_a^*\chi_\omega - \chi_\omega,\> [0\times \mathbb T^{k-1}] \ra &= \la \chi_{\omega}, \> [a \times \mathbb T^{k-1}] - [0 \times \mathbb T^{k-1}]\ra \\
&= \int_{[0,a]\times \mathbb T^{k-1}} \omega \>=\> a \mod \L.
\end{align*}
For $a\notin \L$, it follows that $f_a^*(\chi_\omega) \neq \chi_{f_a^*\omega}$.  Hence, $\omega \mapsto \chi_{\omega}$ is not functorial.

Similarly, let $[c] \in H^k(\mathbb T^k;\L)$ be the standard generator.  Since $f_a$ is homotopic to the identity map, then $f_a^*[c]=[c]$.  Assume that $\chi_c \in \H[\mathbb T^k]{k}{\R}{\L}$ is a differential character, with curvature $\omega$, asociated to $[c]$.  We do not need to assume that $\omega$ is the volume form.  The same argument as above shows that
\[ \la f_a^*\chi_c - \chi_c, \> [0 \times \mathbb T^{k-1}] \ra = \int_{[0,a]\times \mathbb T^{k-1}} \omega \mod \L.\]
The above integral must be non-zero for some value of $a$, since
\[ \lim_{a\to 1} \int_{[0,a]\times \mathbb T^{k-1}} \omega = \int_{\mathbb T^k}\omega = 1.\]
Therefore, $f_a^*\chi_c \neq \chi_c = \chi_{f_a^*c}$, and the right splitting of \eqref{eq:exact2} cannot be functorial.

Finally, the volume form $\omega = dx^1 \wedge \ldots \wedge dx^k$ equals the generator of $H^k(\mathbb T^k;\L)$ in $H^k(\mathbb T^k;\R)$.  We can combine the above two arguments to show that $f^*_a \chi_{c,\omega} \neq \chi_{c,\omega} = \chi_{f_a^*c, f_a^*\omega}$ for any construction $A^k(-;\L) \to \H[-]{k}{\R}{\L}$.  Therefore, \eqref{eq:exact3} can not split naturally.
\end{proof}\end{prop}

\begin{cor}There do not exist right inverses $\T{-} \from \Tc{-}$, $\T{-} \from \Tt{-}$, or $\T{-} \from \Tt{-} \times_{\Th{-}} \Tc{-}$ that are natural with respect to Riemannian isometries when $V\neq 0$.
\begin{proof}By the equivariance from Theorem \ref{thm:torsorsquare}, a right inverse
\[ \xymatrix{ \T{-} \ar[r] & \Tc[\l]{-} \ar@{..>}@/_1pc/[l] }\]
would induce a functorial splitting of the short exact sequence \eqref{eq:exact1}.  But, this was proved to not exist in Proposition \ref{prop:NoSplittings}.  For the same reasons, there can not exist functorial right inverses $\T{-} \from \Tt{-}$ or $\T{-} \from \Tt{-} \times_{\Th{-}} \Tc{-}$.
\end{proof}\end{cor}

\subsection{Compatibility with products}\label{subsec:Product}
There is a natural product in $\check{H}^*$, and it is defined at the cochain level in \cite{HS05}.  For $\L_i \subset V_i$, the homomorphism
\[ \Cbg{p_1}{q_1}{V_1}{\L_1} \otimes \Cbg{p_2}{q_2}{V_2}{\L_2} \to \Cbg{p_1+p_2}{q_1+q_2}{V_1 \otimes V_2}{\L_1 \otimes \L_2} \]
is defined by the formula
\[ \x_1 \cdot \x_2 \= \Big( c_1 \cup c_2 \>,\>  (-1)^{|\omega_1|} c_1 \cup h_2 + h_1 \cup \omega_2 + B(\omega_1, \omega_2) \> , \> \omega_1 \wedge \omega_2  \Big), \]
where $\L_1 \otimes \L_2 \subset V_1\otimes V_2$.  Here $\x_i = (c_i, h_i,\omega_i)$, and $B(\bullet,\bullet)$ is any natural chain homotopy between the cup product $\cup$ and the wedge product $\wedge$.  The product is a graded derivation at the level of cochains,
\begin{equation}\label{eq:dga} d( \x_1 \cdot \x_2) = d\x_1 \cdot \x_2 + (-1)^{|\x_1|} \x_1 \cdot d\x_2  \in \Cbg{*}{*}{V_1\otimes V_2}{\L_1\otimes \L_2},
\end{equation}
so it descends to a homomorphism 
\[ \H{k_1}{V_1}{\L_1} \otimes \H{k_2}{V_2}{\L_2} \to \H{k_1+k_2}{V_1\otimes V_2}{\L_1\otimes \L_2}.\]
In general, this makes $\H{*}{V}{\L}$ into a (left or right) module over $\H{*}{\R}{\Z}$.  When $V=\R$ or $\C$ and $\L\subset V$ is a subring, this gives $\H{*}{V}{\L}$ a graded commutative ring structure.  

For general $\L\subset V$, the product
\begin{equation}\label{eq:0product} \H{k_1}{\R}{0} \otimes \H{k_2}{\R}{\L} \to \H{k_1+k_2}{\R}{0}\end{equation}
can be described using the isomorphism $\H{k}{\R}{0} \iso \Omega^{k-1}(M;V)\big/ d\Omega^{k-2}(M;V)$ from \eqref{eq:IsoTopH0}.  To an equivalence class of $(k-1)$-forms $[\eta_1] \in \H{k_1}{\R}{0}$, and a character $\chi_2 \in \H{k_2}{\R}{\L}$ with curvature $\omega_2$, the product is the equivalence class
\[ [\eta_1 \wedge \omega_2 ] \in \H{k_1+k_2}{\R}{0}.\]

It follows immediately that trivializations and the torsor square from Theorem \ref{thm:torsorsquare} are compatible with the various natural product structures.

\begin{prop}Let $\x_i =(c_i,h_i,\omega_i) \in \H{k_i}{V}{\L_i}$ for $k=1,2$.  If $\s=(b,k,\eta) \in  \cT{\x_1}$, then
\[ \s \cdot \x_2 \in  \cT{\x_1\cdot \x_2}. \]
The characteristic cocycle of $\s \cdot \x_2$ is $b\cup c_2$, the curvature is $\eta \wedge \omega_2$, and the de Rham class is given by the product \eqref{eq:0product}.
\begin{proof}
Using equation \eqref{eq:dga} and the assumption that $d\s = \x_1$, we have
\[ d(\s \cdot \x_2) = d\s \cdot \x_2 +(-1)^k \s \cdot d\x_2 = \x_1 \cdot \x_2.\]
The characteristic cocycle and curvature follow from the definition of $\s\cdot \x_2$.  The de Rham class is the equivalence class of the curvature $\eta \wedge \omega_2$ modulo $d\Omega^{k_1+k_2-2}(M;V)$, which is precisely the product \eqref{eq:0product}.
\end{proof}\end{prop}


\subsection{Compatibility with integration}\label{subsec:Integration}
Hopkins--Singer also define the notion of an $\check{H}$-orientation for a smooth map $f:X \to B$ of relative dimension $n$.  The details can be found in Sections 2.4 and 3.4 of \cite{HS05}, but it amounts to factoring $f$ through an embedding $X \into B\times \R^N$ and choosing a differential Thom cocycle for the normal bundle.  Such a cocycle (with fiberwise compact support) $\check{U} \in \cHc[\nu]{N-n}{\R}{\Z}$ is a differential refinement of a Thom class for the normal bundle $\nu \to X$.  When $f:X\to B$ is a smooth fiber bundle whose fibers are diffeomorphic to a closed manifold $M$, an orientation and Riemannian metric on $M$ define an $\check{H}$-orientation up to isomorphism.

Given an $\check{H}$-oriented map $f:X \to B$, of relative dimension $n$, the usual pushforward construction in cohomology extends to the differential cochain complex, giving the homomorphisms
\[ \Cbg[X]{p}{q}{V}{\L} \overset{\pi^* \cup \check{U}}\longrightarrow \Cbgc[\nu]{p+N-n}{q+N-n}{V}{\L}\to \Cbgc[B\times \R^N]{p+N-n}{q+N-n}{V}{\L} \overset{\int_{\R^N}}\longrightarrow \Cbg[B]{p-n}{q-n}{V}{\L}.  \]
Here, $\int_{\R^N}$ is defined using the slant product on singular cochains and integration on forms.  The composition, in both the differential cochain and singular cochain complex, is referred to as integration and is denoted
\begin{align}\label{eq:int1} \int_{X/B}: \Cbg[X]{p}{q}{\R}{\L} &\longrightarrow \Cbg[B]{p-n}{q-n}{\R}{\L}, \\
\label{eq:int2} \int_{X/B}: C^q(X;\L) &\longrightarrow C^{q-n}(B;\L).
\end{align}
This is compatible with the usual integration of differential forms; i.e. if $\x \in \cH[X]{k}{\R}{\L}$, then
\begin{equation}\label{eq:IntegrationCurvature}
\curv \left( \int_{X/B} \x \right) = \int_{X/B} \curv(\x) \quad \in \Omega^{k-n}(B;V),
\end{equation}
where the right-hand integral takes place in the de Rham complex.

The integration functor satisfies a generalized Stokes formula
\begin{equation}\label{eq:Stokes} d \int_{X/B} \x = \int_{X/B} d\x \quad +\quad (-1)^{|\x|-n} \int_{\d X/B} \x.
\end{equation}
Assuming the fibers of $f$ have no boundary, the Stokes formula implies $\int_{X/B}$ is a cochain map and induces homomorphisms
\[ \int_{X/B} : \H[X]{k}{V}{\L} \longrightarrow \H[B]{k-n}{V}{\L}\]
compatible with the usual pushforward in cohomology.  When $\L=0$, this homomorphism is given by integrating equivalence classes of forms:
\begin{equation}\label{eq:int3} \int_{X/B}: \H[X]{k}{V}{0} \iso \frac{\Omega^{k-1}(X;V)}{d\Omega^{k-2}(X;V)}  \overset{\int_{X/B}}\longrightarrow \frac{\Omega^{k-n-1}(B;V)}{d\Omega^{k-n-2}(B;V)} \iso \H[B]{k-n}{V}{0}.\end{equation}
It is straightforward to check that trivializations behave well with respect to integration.

\begin{prop}\label{prop:Integration}Suppose that $f:X \to B$ is an $\check{H}$-oriented map of relative dimension $n$ and whose fibers have no boundary.  If $\x \in \cH[X]{k}{V}{\L}$ admits a trivialization, integration induces a functor
\[ \int_{X/B}\>:\> \cT{\x} \longrightarrow \cT{\int_{X/B} \x} \]
that is compatible with the various forms of integration.  More precisely, the following diagram containing the two torsor squares commutes; the diagonal arrows are induced by the integrations \eqref{eq:int1}, \eqref{eq:int2}, \eqref{eq:int3}, and the ordinary integration of differential forms.
\[ \xymatrix@=3mm{\T{\x} \ar[rr] \ar[dd] \ar[dr] && \Tt{\x} \ar'[d][dd] \ar[dr] \\
 &\T{\int \x} \ar[dd] \ar[rr]  && \Tt{\int \x} \ar[dd] \\
 \Tc{\x} \ar'[r][rr] \ar[dr] && \Th{\x} \ar[dr] \\
 & \Tc{ \int \x} \ar[rr] && \Th{\int \x} } \]
\begin{proof}
The Stokes formula implies that if $d\s = \x$, then $d\int_{X/B}\s = \int_{X/B}\x$.  If $d\y = \s_2 - \s_1$, then $d \int_{X/B} \y = \int_{X/B} \s_2 - \int_{X/B} \s_1$.  Therefore, the integration in \eqref{eq:int1} induces a functor $\cT{\x} \to \cT{\int_{X/B} \x}$.  The commutativity of the diagram is an immediate consequence of the the fact that integration in the differential cochain complex commutes with integration in the singular and de Rham complexes.
\end{proof}\end{prop}


\section{Trivializing circle bundles}\label{sec:Examples}

The category $\cH{2}{\R}{\Z}$ is equivalent to the category of principal $\R/\Z$-bundles with connection.  We now consider the relationship between trivializations of principal $\R/\Z$-bundles and the torsor square for trivializations of degree 2 classes.

Begin by fixing a principal $\R/\Z$-bundle $P\overset{\pi}\to M$ with connection $\Theta$.  A trivialization of $P$, not necessarily preserving the connection, is defined as a bundle isomorphism $P\iso M\times \R/\Z$; this is naturally equivalent to a global section $p \in C^\infty(M,P)$.  The bundle $P$ is trivializable if and only if its characteristic class $c(P)=0 \in H^2(M;\Z)$.  Assuming $c(P) = 0$, the set of global sections naturally fits into the following commutative square.
\begin{equation}\label{eq:SquareU(1)Bun} \xymatrix{ C^\infty(M,P) \ar[r] \ar[d]&C^\infty(M,P)_{/_\text{htpy}} \ar[d] \\
\{ p^*\Theta \in \Omega^1(M) \st p \in C^\infty(M,P) \} \ar[r] & \{ \text{Lift of }\hol_\Theta \text{to }\R\} } \end{equation}
The top right corner consists of global sections modulo homotopy through the space of sections.  The bottom left corner consists of all globally defined connection 1-forms on $M$ for the connection $\Theta$.  The bottom right corner consists of all lifts of the holonomy from $\R/\Z$ to $\R$.

The group $C^\infty(M,\R/\Z)$, with product given by pointwise addition in $\R/\Z$, acts freely and transitively on the space of global sections.  For $p \in C^\infty(M,P)$ and $f \in C^\infty(M,\R/\Z)$, the resulting section is defined by
\[ \big(p+f \big)(x) \= p(x) + f(x).\]
Pulling back the connection along the section $p+f$ results in the change
\[ (p+f)^*\Theta = p^*\Theta + df.\]
(We do not need to include the Adjoint action because $\R/\Z$ is abelian.)

Any two lifts of $\hol_\Theta$ to $\R$ will differ by an element in $\Hom(H_1(M),\Z)$.  A priori, one may think they only need to differ by an element in $\Hom(Z_1(M),\Z)$.  We assume, though, that our lifts are ``smooth" in the sense that integrating the curvature form calculates the difference between the $\R$-valued holonomies of homotopic paths.  More concisely, the lift is an element in $\H{2}{\R}{0}$.

It follows that the commutative square \eqref{eq:SquareU(1)Bun} is a torsor over the following commutative square, which was explained in Example \ref{ex:functions} to be canonically isomorphic to the standard commutative square for $\H{1}{\R}{\Z}$.
\[ \xymatrix@C=5mm{ \H{1}{\R}{\Z} \ar@{}[r]|(.30){\iso} & C^\infty(M,\R/\Z) \ar[r] \ar[d]& C^\infty(M, \R/\Z)_{/_\text{htpy}} \ar[d] \ar@{}[r]|(.55){\iso}&H^1(M;\Z)\\
 \Omega^1(M)_\Z \ar@{}[r]|(.30){\iso}& \{ df  \st f \in C^\infty(M, \R/\Z)\} \ar[r]& \Hom(H_1(M), \Z) \ar@{}[r]|(.55){\iso}& H^1(M;\R)_\Z.} \]
The bottom right isomorphism was proven in Lemma \ref{lem:uct}.

We wish to show that the square \eqref{eq:SquareU(1)Bun} is equivalent to the square from Theorem \ref{thm:torsorsquare}.  We should first explain the equivalence, discussed in Example 2.7 of \cite{HS05}, between the category of principal $\R/\Z$-bundles with connection on $M$ and the category $\cH{2}{\R}{\Z}$.  The key idea is that a bundle is determined completely by its local sections, and the connection $\Theta$ is determined by local connection forms $\nabla$.

To any $\R/\Z$-bundle with connection $(P,\Theta) \overset{\pi}\to M$, one can associate its sheaf of sections: to the open set $U \subset M$ is assigned $C^\infty(U, P)$, which is a torsor for $C^\infty(U, \R/\Z)$ when it is non-empty.  The connection $\Theta$ determines local connection 1-forms, or maps
\[ \nabla: C^\infty(U,P) \to \Omega^1(U) \]
that transform according to the previously mentioned rule $\nabla^{p+f} = \nabla^{p} + df$.  A morphism of $(P,\Theta)$ induces a morphism of the resulting sheaf $C^\infty(-,P)$ that preserves the connection $\nabla$.

Similarly, to a differential cocycle $\x \in \cH{2}{\R}{\Z}$  one can consider the sheaf of trivializations
\[ U \mapsto \T{\x\big|_{U}}.\]
When $\T{\x\big|_{U}}$ is non-empty, it is a torsor for $\H[U]{1}{\R}{\Z} \iso C^\infty(U, \R/\Z)$, and $\T{\x\big|_{U}} \to \Tc{\x\big|_{U}}$ gives a local connection form $\nabla$ transforming by $\nabla^{\s + f} = \nabla^{\s} + df$.

Now, we assume that $(P,\Theta) \overset{\pi}\to M$ is represented by $\x \in \cH{2}{\R}{\Z}$ in the sense that the resulting sheaf of sections with connection are equal, or that we have chosen an isomorphism between them.

\begin{prop}\label{prop:degree2}If $(P,\Theta) \overset{\pi}\to M$ is represented by $\x \in \cH{2}{\R}{\Z}$, then the commutative square \eqref{eq:SquareU(1)Bun} is in natural bijection with the commutative square
\begin{equation} \label{eq:SquareC1}\xymatrix{ \T{\x} \ar[r] \ar[d] &\Tt{\x} \ar[d]\\
\Tc{\x} \ar[r] & \Th{\x}. }\end{equation}
\begin{proof}
We assume $c(P)=0$, since all of the relevant spaces are empty otherwise.  An element in $\T{\x}$ is precisely the choice of a global section of the sheaf 
\[ \big( U\mapsto \T{\x\big|_{U}} \big) \iso \big( U \mapsto C^\infty(U, P) \big),\]
which is equivalent to a global section $p \in C^\infty(M,P)$.  The sheaves of local sections/trivializations are equivariant with respect to the action of $C^\infty(-, \R/\Z) \iso \H[-]{1}{\R}{\Z}$, so the bijection
\[ \T{\x} \longleftrightarrow C^\infty(M,P) \]
is also equivariant.  The squares  \eqref{eq:SquareU(1)Bun} and \eqref{eq:SquareC1} are both torsors over the square for $\H{1}{\R}{\Z}$, and we have an equivariant map between the top-left corner of both squares.  While bijections between the remaining corners can be constructed explicitly, applying Lemma \ref{lem:InducedTorsorMaps} immediately implies they uniquely exist.
\end{proof}\end{prop}

Finally, we remark that Propositions \ref{prop:Integration} and \ref{prop:degree2} combine in a useful way.  Suppose that 
\[ X \overset{f}\longrightarrow B\]
is an $\check{H}$-oriented map of relative dimension $n$, and that we have a principal $\R/\Z$-bundle with connection $(P,\Theta)\to B$.  If we can represent this bundle by $\int_{X/B} \x$ for some $\x \in \cH[X]{n+2}{\R}{\Z}$,  it may be more convenient to trivialize the bundle on $B$ by trivializing $\x$ in $X$.

We first make a couple of remarks about trivializing $\R/\Z$-bundles with connection.  There is no torsion in the cohomology group $H^1(B;\Z)$, so the homomorphism $H^1(B;\Z) \to H^1(B;\R)_\Z$ is an isomorphism.  This implies that lifting the holonomy to $\R$ is equivalent to specifying the homotopy class of a global section.  Also, the choice of a global connection form on $B$ determines a trivial connection on $P$ by $\Theta -\pi^* p^* \Theta \in \Omega^1(P)$.  On each connected component of $B$, the choice of a trivial connection only determines a global section up to a constant $\R/\Z$-phase.  Algebraically, this is explained by the fact that $\H[B]{1}{\R}{\Z} \to \Omega^1(B)_\Z$ has non-trivial kernel $H^0(B;\R/\Z)$.

Assume that $\int_{X/B} \x$ represents $(P,\Theta) \overset{\pi}\to B$.  By Propositions \ref{prop:Integration} and \ref{prop:degree2}, a trivialization $\s \in \T{\x}$ induces a global section of $P$ by $\int_{X/B} \s \in \T{\int_{X/B}\x}$.  The characteristic class of $\s$ determines the homotopy class of the induced global section, the curvature of $\s$ induces a trivial connection on $P$, and the de Rham class of $\s$ lifts the holonomies of $(P,\Theta)$ to $\R$.  If $H^n(X;\R/\Z)=0$, then $\T{\x} \to \Tc{\x}$ is a bijection, so the global section of $P$ can be completely determined by a differential form on $X$.  For reasons such as this, it is sometimes easier to work with trivializations in the space $X$.

\begin{lemma}\label{lem:InducedTorsorMaps}Let $0 \to G_0 \to G_1 \to G_2 \to 0$ be a short exact sequence of groups.  Suppose $A_i$ and $B_i$ are (left) torsors over the groups $G_i$, and there are $G_1 \to G_2$ equivariant maps $A_1 \overset{\alpha}\to A_2$ and $B_1 \overset{\beta}\to B_2$.  Then, a $G_1$-equivariant map $A_1 \overset{f_1}\to B_1$ induces a unique $G_2$ equivariant map $A_2 \overset{f_2}\to B_2$.
\[ \xymatrix{A_1 \ar[r]^{\alpha}\ar[d]^{f_1}&A_2\ar@{-->}[d]^{f_2} \\
B_2 \ar[r]^{\beta}& B_2}\]
\begin{proof}
For $a_2 \in A_2$, choose some element $a_1 \in \alpha^{-1}(a_2) \subset A_1$.  Define $f_2(a_2) \= \beta\left( f_1(a_1)\right)$.  This is well-defined since any other choice $a_1' = g_0 a_1$ for a unique $g_0 \in G_0 = \Ker (G_1 \to G_2)$, and 
\[ \beta\left( f_1(g_0 a_1) \right) = \beta \left( g_0 f_1(a_1) \right) = \beta \left(f_1(a_1) \right). \]
The $G_2$-equivariance of $f_2$ follows by a similar argument.  If $g_1 \mapsto g_2$, then
\[ f_2(g_2 a_2) = \beta( f_1( g_1 a_1)) = g_2 f_2(a_2).\qedhere\]
\end{proof} \end{lemma}


\bibliographystyle{alphanum}
\bibliography{MyBibDesk}

\end{document}

%% file: elliptic-macros-thesis.tex
\newcommand\hol{\operatorname{hol}}

\newcommand{\x}{\pmb{x}}
\newcommand{\y}{\pmb{y}}

\def\myPOS#1{\POS}

\def\C{\mathbb C}

\def\H{\mathbb H}

\def\Q{\mathbb Q}
\def\R{\mathbb R}
\def\T{\mathbb T}
\def\Z{\mathbb Z}

\def\cH{\mathcal H}

\def\cL{\mathcal L}

\def\cT{\mathcal T}

\def\cV{\mathcal V}

\def\cT{\mathcal T}

\newcommand{\ra}{\longrightarrow}
\newcommand{\la}{\leftarrow}

\newcommand{\into}{\hookrightarrow}

\def\wt{\widetilde}

\def\={:=}
\def\doublearrow#1#2
{\quad\raisebox{.1cm}{$\overset{#1}\lra$}  
\hskip -.67cm
\raisebox{-.1cm}{$\underset {#2}\lra$}\quad}

\def\Ext{\operatorname{Ext}}

\def\Aut{\operatorname{Aut}}

\def\Ker{\operatorname{Ker}}

\def\d{\partial}

\def\Mor{\operatorname{Mor}}

\def\Hom{\operatorname{Hom}}

\def\nb-{\nobreakdash-}

\def\:{\colon}
\def\iso{\cong}
\def\-1{^{-1}}
\def\o{\circ}